\documentclass[11pt,english]{article}
\usepackage[T1]{fontenc}
\usepackage[latin9]{inputenc}
\usepackage{amsmath}
\usepackage{amssymb}
\usepackage{amsthm}
\usepackage{amsfonts}
\usepackage{babel}
\usepackage{fullpage}
\usepackage{color}
\usepackage{times}

\parskip=\medskipamount
\numberwithin{equation}{section}

\newtheorem{theorem}{Theorem}[section]
\newtheorem{lem}[theorem]{Lemma}
\newtheorem{fact}[theorem]{Fact}

\newtheorem{claim}[theorem]{Claim}
\newtheorem{obs}[theorem]{Observation}
\newtheorem{prop}[theorem]{Proposition}

\theoremstyle{definition}
\newtheorem{defn}{Definition}[section]

\theoremstyle{remark}

\newcommand{\Exp}{\operatornamewithlimits{\mathbb{E}}}
\newcommand{\cube}{\operatorname{\{0, 1\}}}

\newcommand{\eps}{\epsilon}
\newcommand{\polyn}{\operatorname{poly} }
\newcommand{\calf}{{\cal F}}
\newcommand{\calm}{{\cal M}}
\newcommand{\const}{\mathbf{Const}}
\newcommand{\mlin}{\mathbf{Lin}}
\newcommand{\aff}{{\mathbf{Aff}}}
\newcommand{\blin}{\overline{\mathbf{Lin}}}
\newcommand{\baff}{\overline{\mathbf{Aff}}}
\newcommand{\flin}{{\calf}_{\mathbf{lin}}}
\newcommand{\faff}{{\calf}_{\mathbf{aff}}}
\newcommand{\bflin}{\overline{{\calf}_{\mathbf{lin}}}}
\newcommand{\bfaff}{\overline{{\calf}_{\mathbf{aff}}}}

\begin{document}
\title{Testing Linear-Invariant Non-Linear Properties}
\author{Arnab Bhattacharyya%
\thanks{MIT CSAIL. \texttt{abhatt@csail.mit.edu}. Research supported in part
by a DOE Computational Science Graduate Fellowship.%
} \and Victor Chen%
\thanks{MIT CSAIL. \texttt{victor@csail.mit.edu}. Research supported in part
by NSF Award CCR-0514915.%
} \and Madhu Sudan%
\thanks{MIT CSAIL. \texttt{madhu@csail.mit.edu}. Research supported in part
by NSF Award CCR-0514915.%
} \and Ning Xie%
\thanks{MIT CSAIL. \texttt{ningxie@csail.mit.edu}. Research supported in part
by an Akamai Presidential Fellowship and NSF grant 0514771.%
} }

\date{}

\maketitle
\setcounter{page}{0}
\begin{abstract}
We consider the task of testing properties of Boolean functions that
are invariant under linear transformations of the Boolean cube. Previous
work in property testing, including the linearity test and the test
for Reed-Muller codes, has mostly focused on such tasks for linear
properties. The one exception is a test due to Green for {}``triangle
freeness'': a function $f:\cube^{n}\to\cube$ satisfies this property
if $f(x),f(y),f(x+y)$ do not all equal $1$, for any pair $x,y\in\cube^{n}$.

Here we extend this test to a more systematic study of testing for
linear-invariant non-linear properties. We consider properties that
are described by a single forbidden pattern (and its linear transformations),
i.e., a property is given by $k$ points $v_{1},\ldots,v_{k}\in\cube^{k}$
and $f:\cube^{n}\to\cube$ satisfies the property that if for all
linear maps $L:\cube^{k}\to\cube^{n}$ it is the case that $f(L(v_{1})),\ldots,f(L(v_{k}))$
do not all equal $1$. We show that this property is testable if the
underlying matroid specified by $v_{1},\ldots,v_{k}$ is a graphic
matroid. This extends Green's result to an infinite class of new properties.

Our techniques extend those of Green and in particular we establish
a link between the notion of ``$1$-complexity linear systems''
of Green and Tao, and graphic matroids, to derive the results. 
\end{abstract}
\newpage{}

\pagenumbering{arabic}

\section{Introduction\label{sec:Introduction}}

Property testing considers the task of testing, {}``super-efficiently'',
if a function $f:D\to R$ mapping a finite domain $D$ to a finite
range $R$ essentially satisfies some desirable property. Letting
$\{D\to R\}$ denote the set of all functions from $D$ to $R$, a
{\em property} is formally specified by a family $\calf\subseteq\{D\to R\}$
of functions. A {\em tester} has oracle access to the function
$f$ and should accept with high probability if $f\in\calf$ and reject
(also with high probability) functions that are {\em far} from
$\calf$, while making very few queries to the oracle for $f$. Here,
distance between functions $f,g:D\to R$, denoted $\delta(f,g)$,
is simply the probability that $f(x)\ne g(x)$ when $x$ is chosen
uniformly at random from $D$ and $\delta(f,\calf)=\min_{g\in\calf}\{\delta(f,g)\}$.
We say $f$ is $\delta$-far from $\calf$ if $\delta(f,\calf)\geq\delta$
and $\delta$-close otherwise. The central parameter associated with
a tester is the number of oracle queries it makes to the function
$f$ being tested. In particular, a property is called \emph{(locally)
testable} if there is a tester with query complexity that is a constant
depending only on the distance parameter $\delta$. Property testing
was initiated by the works of Blum, Luby and Rubinfeld~\cite{BLR}
and Babai, Fortnow and Lund~\cite{BFL} and was formally defined
by Rubinfeld and Sudan~\cite{RubSud96}. The systematic exploration
of property testing was initiated by Goldreich, Goldwasser, and Ron~\cite{GGR}
who expanded the scope of property testing to combinatorial and graph-theoretic
properties (all previously considered properties were algebraic).
In the subsequent years, a rich collection of properties have been
shown to be testable~\cite{AlonShapira1,AlonShapira2,AFNS,BCLSSV,ParRonSam,AKNS,AKKLR,KR04,JPRZ04}
and many property tests have ended up playing a crucial role in constructions
of probabilistically checkable proofs~\cite{AroSaf,ALMSS,BelGolSud98,Has97,SamTre06}.

The rich collection of successes in property testing raises a natural
question: Why are so many different properties turning out to be locally
testable? Are there some broad {}``features'' of properties that
make them amenable to such tests? Our work is part of an attempt to
answer such questions. Such questions are best understood by laying
out broad (infinite) classes of properties (hopefully some of them
are new) and showing them to be testable (or characterizing the testable
properties within the class). In this paper we introduce a new such
class of properties, and show that (1) they are locally testable,
and (2) that they contain infinitely many new properties that were
not previously known to be testable.

\paragraph{The properties, and our results:}

The broad scope of properties we are interested in are properties
that view their domain $D$ as a vector space and are invariant under
linear transformations of the domain. Specifically, we consider the
domain $D=\cube^{n}$, the vector space of $n$-dimensional Boolean
vectors, and the range $R=\cube$. In this setting, a property $\calf$
is said to be {\em linear-invariant} if for every $f\in\calf$
and linear map $L:\cube^{n}\to\cube^{n}$ we have that $f\circ L\in\calf$.
Specific examples of linear-invariant properties that were previously
studied (esp. in the Boolean setting) include that of linearity, studied
by Blum et al.~\cite{BLR} and Bellare et al.~\cite{BCHKS}, and
the property of being a {}``moderate-degree'' polynomial (aka Reed-Muller
codeword) studied by Alon et al.~\cite{AKKLR}%
\footnote{In the literature, the term low-degree polynomial is typically used
for polynomials whose degree is smaller than the field size. In the
work of \cite{AKKLR} the degrees considered are larger than the field
size, but are best thought of as large constants. The phrase {}``moderate-degree''
above describes this setting of parameters.%
}. While the tests in the above mentioned works potentially used all
features of the property being tested, Kaufman and Sudan~\cite{KaufmanSudan08}
show that the testability can be attributed principally to the linear-invariance
of the property. However their setting only considers {\em linear}
properties, i.e., $\calf$ itself is a vector space over $\cube$
and this feature plays a key role in their results: It lends an algebraic
flavor to all the properties being tested and plays a central role
in their analysis.

We thus ask the question: Does linear-invariance lead to testability
even when the property $\calf$ is not linear? The one previous work
in the literature that gives examples of non-linear linear-invariant
properties is Green~\cite{Gre03} where a test for the property of
being {}``triangle-free'' is described. A function $f:\cube^{n}\to\cube$
is said to be {\em triangle-free} if for every $x,y\in\cube^{n}$
it is the case that at least one of $f(x),f(y),f(x+y)$ does not equal
$1$. The property of being triangle-free is easily seen to be linear-invariant
and yet not linear. Green~\cite{Gre03} shows that the natural test
for this property does indeed work correctly, though the analysis
is quite different from that of typical algebraic tests and is more
reminiscent of graph-property testing. In particular, Green develops
an algebraic regularity lemma to analyze this test. (We note that
the example above is not the principal objective of Green's work,
which is directed mostly at abelian groups $D$ and $R$. The above
example with $D=\cube^{n}$ and $R=\cube$ is used mainly as a motivating
example.)

Motivated by the above example, we consider a broad class of properties
that are linear-invariant and non-linear. A property in our class
is given by $k$ vectors $v_{1},\ldots,v_{k}$ in the $k$-dimensional
space $\cube^{k}$. (Throughout this paper we think of $k$ as a constant.)
These $k$ vectors uniformly specify a family $\calf=\calf_{n;v_{1},\ldots,v_{k}}$
for every positive integer $n$, containing all functions that, for
every linear map $L:\cube^{k}\to\cube^{n}$ take on the value $0$
on at least one of the points $L(v_{1}),\ldots,L(v_{k})$. (In Appendix~\ref{sec:Non-monotone-properties}
we consider an even more generalized class of properties where the
forbidden pattern of values for $f$ is not $1^{k}$ but some other
string and show a limited set of cases where we can test such properties.)
To see that this extends the triangle-freeness property, note that
triangle-freeness is just the special case with $k=3$ and $v_{1}=\langle100\rangle$,
$v_{2}=\langle010\rangle$, $v_{3}=\langle110\rangle$. Under different
linear transforms, these three points get mapped to all the different
triples of the form $x,y,x+y$ and so $\calf_{n;v_{1},v_{2},v_{3}}$
equals the class of triangle-free functions.

Before giving a name to our class of functions, we make a quick observation.
Note that the property specified by $v_{1},\ldots,v_{k}$ is equivalent
to the property specified by $T(v_{1}),\ldots,T(v_{k})$ where $T$
is a non-singular linear map from $\cube^{k}\to\cube^{k}$. Thus the
property is effectively specified by the dependencies among $v_{1},\ldots,v_{k}$
which are in turn captured by the matroid%
\footnote{For the sake of completeness we include a definition of matroids in
Appendix~\ref{sec:def-matroid}. However a reader unfamiliar with
this notion may just use the word matroid as a synonym for a finite
collection of binary vectors, for the purposes of reading this paper. %
} underlying $v_{1},\ldots,v_{k}$. This leads us to our nomenclature:

\begin{defn}\label{def:matfree} Given a (binary, linear) matroid
$\calm$ represented by vectors $v_{1},\ldots,v_{k}\in\cube^{k}$,
the property of being {\em $\calm$-free} is given by, for every
positive integer $n$, the family \[
\calf_{\calm}=\{f:\cube^{n}\to\cube|\forall~{\rm linear}~L:\cube^{k}\to\cube^{n},\langle f(L(v_{1})),\ldots,f(L(v_{k}))\rangle\ne1^{k}\}.\]
 \end{defn}

The property of being $\calm$-free has a natural $k$-local test
associated with it: Pick a random linear map $L:\cube^{k}\to\cube^{n}$
and test that $\langle f(L(v_{1})),\ldots,f(L(v_{k}))\rangle\ne1^{k}$.
Analyzing this test turns out to be non-trivial, and indeed we only
manage to analyze this in special cases.

Recall that a matroid $\calm=\{v_{1},\ldots,v_{k}\}$, $v_{i}\in\cube^{k}$,
forms a {\em graphic matroid} if there exists a graph $G$ on $k$
edges with the edges being associated with the elements $v_{1},\ldots,v_{k}$
such that a set $S\subseteq\{v_{1},\ldots,v_{k}\}$ has a linear dependency
if and only if the associated set of edges contains a cycle. In this
paper, we require that the graph $G$ be simple, that is, without
any self-loops or parallel edges. Our main theorem shows that the
property $\calf$ associated with a graphic matroid $v_{1},\ldots,v_{k}\in\cube^{k}$
is testable.

\begin{theorem} \label{thm:main1} For a graphic matroid $\calm$,
the property of being $\calm$-free is locally testable. Specifically,
let $\calm=\{v_{1},\ldots,v_{k}\}$ be a graphic matroid. Then, there
exists a function $\tau:\mathbb{R}^{+}\rightarrow\mathbb{R}^{+}$
and a $k$-query tester that accepts members of ${\calm}$-free functions
with probability one and rejects functions that are $\epsilon$-far
from being ${\calm}$-free with probability at least $\tau(\eps)$.
\end{theorem}

Our bound on $\tau$ is quite weak. We let $W(t)$ denote a tower
of twos with height $\left\lceil t\right\rceil $. Our proof only
guarantees that $\tau(\eps)\geq W(\polyn(1/\eps))^{-1}$, a rather
fast vanishing function. We do not know if such a weak bound is required
for any property we consider.

We describe the techniques used to prove this theorem shortly (which
shed light on why our bound on $\tau$ is so weak) but first comment
on the implications of the theorem . First, note that for a graphic
matroid it is more natural to associate the property with the underlying
graph. We thus use the phrase $G$-free to denote the property of
being $\calm$-free where $\calm$ is the graphic matroid of $G$.
This terminology recovers the notion of being triangle-free, as in
\cite{Gre03}, and extends to cover the case of being $k$-cycle free
(also considered in \cite{Gre03}). But it includes every other graph
too!

Syntactically, Theorem~\ref{thm:main1} seems to include infinitely
many new properties (other than being $k$-cycle free). However, this
may not be true semantically. For instance the property of being triangle-free
is essentially the same as being $G$-free for every $G$ whose biconnected
components are triangles. Indeed, prior to our work, it was not even
explicitly noted whether being $C_{k}$-free is essentially different
from being triangle-free. (By {}``essentially'', we ask if there
exist triangle-free functions that are {\em far} from being $C_{k}$-free.)
It actually requires careful analysis to conclude that the family
of properties being tested include (infinitely-many) new ones. Our
second theorem addresses this point.

\begin{theorem}\label{thm:main2} The class of $G$-free properties
include infinitely many distinct ones. In particular:
\begin{enumerate}
\item For every odd $k$, if $f$ is $C_{k+2}$-free, then it is also $C_{k}$-free.
Conversely, there exist functions $g$ that are $C_{k}$-free but
far from being $C_{k+2}$-free. 
\item If $k\leq\ell$ and $f$ is $K_{k}$-free, then it is also $K_{\ell}$-free.
On the other hand, if $k\geq3$ and $\ell\geq\binom{k}{2}+2$ then
there exists a function $g$ that is $K_{\ell}$-free but far from
being $K_{k}$-free. 
\end{enumerate}
\end{theorem}

\paragraph{Techniques:}

Our proof of Theorem~\ref{thm:main1} is based on Green's analysis
of the triangle-free case~\cite{Gre03}. To analyze the triangle-free
case, Green develops a {}``regularity'' lemma for groups, which
is analogous to Szemerédi's regularity lemma for graphs. In our setting,
Green's regularity lemma shows how, given any function $f:\cube^{n}\to\cube$,
one can find a subgroup $H$ of $\cube^{n}$ such that the restriction
of $f$ to almost all cosets of $H$ is {}``regular'', where {}``regularity''
is defined based on the {}``Fourier coefficients'' of $f$. (These
notions are made precise in Section~\ref{sub:Fourier-Analysis}.)

This lemma continues to play a central role in our work as well, but
we need to work further on this. In particular, a priori it is not
clear how to use this lemma to analyze $\calm$-freeness for {\em
arbitrary} matroids $\calm$. To extract a large feasible class of
matroids we use a notion from a work of Green and Tao~\cite{GreenTao:LinearPrimes}
of the complexity of a linear system (or matroids, as we refer to
them). The {}``least complex'' matroids have complexity 1, and we
show that the regularity lemma can be applied to all matroids of complexity
$1$ to show that they are testable (see Section~\ref{sec:Matroids-testable}).

The notion of a $1$-complex matroid is somewhat intricate, and a
priori it may not even be clear that this introduces new testable
properties. We show (in Section~\ref{sec:Graphic-is-1}) that these
properties actually capture all graphic matroids which is already
promising. Yet this is not a definite proof of novelty, and so in
Section~\ref{sec:Infinitely-properties} we investigate properties
of graphic matroids and give some techniques to show that they are
{}``essentially'' different. Our proofs show that if two (binary)
matroids are not {}``homomorphically'' equivalent (in a sense that
we define) then there is an essential difference between the properties
represented by them.

\paragraph{Significance of problems/results:}

We now return to the motivation for studying $\calm$-free properties.
Our interest in these families is mathematical. We are interested
in broad classes of properties that are testable; and invariance seems
to be a central notion in explaining the testability of many interesting
properties. Intuitively, it makes sense that the symmetries of a property
could lead to testability, since this somehow suggests that the value
of a function at any one point of the domain is no more important
than its values at any other point. Furthermore this intuition is
backed up in many special cases like graph-property testing (where
the family is invariant under all permutations of the domain corresponding
to relabelling the vertex names). Indeed this was what led Kaufman
and Sudan~\cite{KaufmanSudan08} to examine this notion explicitly
in the context of algebraic functions. They considered families that
were linear-invariant and {\em linear}, and our work is motivated
by the quest to see if the latter part is essential.

In contrast to other combinatorial settings, linear-invariance counts
on a (quantitatively) very restricted collection of invariances. Indeed
the set of linear transforms is only quasi-polynomially large in the
domain (which may be contrasted with the exponentially large set of
invariances that need to hold for graph-properties). So ability to
test properties based on this feature is mathematically interesting
and leads to the question: what kind of techniques are useful in these
settings. Our work manages to highlight some of those (in particular,
Green's regularity lemma).

\paragraph{Parallel works:}

After completing our work, we learned from Asaf Shapira that, independently
of us, $\calm$-freeness for an arbitrary matroid $\calm$ has been
shown to be testable in Shapira's recent preprint \cite{Shapira}.
His result solves a question that we posed as open in an earlier version
of this paper. His result is built on the work of Král', Serra, and
Vena in~\cite{KralSerraVena1}, where an alternate proof of Green's
cycle-freeness result is provided. Essentially the authors in~\cite{KralSerraVena1}
demonstrate a reduction from testing freeness of the cycle matroid
in a function to testing freeness of the cycle subgraph in a graph,
and then they apply regularity lemmas for graphs to analyze the number
of cycles in a function far from being cycle-free. In this manner,
the authors show that Theorem~\ref{thm:main1} holds as well. By
extending this method and utilizing hypergraph regularity lemmas,
Shapira~\cite{Shapira} and Král', Serra, and Vena in a followup
work~\cite{KralSerraVena2} show that arbitrary monotone matroid-freeness
properties are testable.

We remark that our proofs are very different from those in~\cite{KralSerraVena1},~\cite{KralSerraVena2},
and~\cite{Shapira}, and in particular, our view on invariance leads
us to develop techniques to show that syntactically different properties
are indeed distinct.

\paragraph{Organization of this paper:}

In the following section (Section~\ref{sec:Definitions}) we define
a slightly broader class of properties that we can consider (including
some non-monotone properties). We also define the notion of 1-complexity
matroids which forms a central tool in our analysis of the tests.
In Section~\ref{sec:Matroids-testable} we show that for any 1-complexity
matroid $\calm$, $\calm$-freeness is testable. In Section~\ref{sec:Graphic-is-1}
we show that graphic matroids are 1-complexity matroids. Theorem~\ref{thm:main1}
thus follows from the results of Section~\ref{sec:Matroids-testable}~and~\ref{sec:Graphic-is-1}.
In Section~\ref{sec:Infinitely-properties} we prove that there are
infinitely many distinct properties among $G$-free properties. Finally,
in Appendix~\ref{sec:Non-monotone-properties}, we include results
on testing some non-monotone properties, along with some {}``collapse''
results showing that many non-monotone properties collapse to some
simple ones in Appendix~\ref{sub:Characterization}.

\section{Additional definitions, results, and overview of proofs\label{sec:Definitions}}

In this section, we describe some further results that we present
in the paper and give an outline of proofs.

\subsection{Extensions to non-monotone families}

We start with a generalization of Definition \ref{def:matfree} to
a wider collection of forbidden patterns.

\begin{defn} Given $\Sigma\in\cube^{k}$ and a binary matroid $\calm$
represented by vectors ~$v_{1},\ldots,v_{k}\in\cube^{k}$, the property
of being {\em $(\calm,\Sigma)$-free} is given by, for every positive
$n$, the family $\calf_{(\calm,\Sigma)}=\{f:\cube^{n}\to\cube|\forall~{\rm linear}~L:\cube^{k}\to\cube^{n},\langle f(L(v_{1})),\ldots,f(L(v_{k}))\rangle\ne\Sigma\}$.
\end{defn}

If for some linear $L:\cube^{k}\to\cube^{n}$, $\langle f(L(v_{1})),\ldots,f(L(v_{k}))\rangle=\Sigma$,
then we say $f$ {\em contains $(\calm,\Sigma)$ at $L$}. Also,
to be consistent with Definition \ref{def:matfree}, we suppress mention
of $\Sigma$ when $\Sigma=1^{k}$.

Recall that a property $\mathcal{P}\subseteq\{D\to\{0,1\}\}$ is said
to be {\em monotone} if $f\in{\cal {P}}$ and $g\prec f$ implies
$g\in{\cal {P}}$, where $g\prec f$ means that $g(x)\leq f(x)$ for
all $x\in D$. \begin{obs} For a binary matroid $\calm$, $(\calm,\Sigma)$-freeness
is a monotone property if and only if $\Sigma=1^{k}$. \end{obs} 

In addition to our main results (Theorems \ref{thm:main1} and \ref{thm:main2})
on monotone properties, we also obtain local testability results for
a limited class of non-monotone properties.

\begin{theorem}\label{thm:nonmontest} Let $C_{k}$ denote the cycle
on $k$ vertices and let $\Sigma$ be an arbitrary element of $\cube^{k}$.
Then, there exists a function $\tau:\mathbb{R}^{+}\rightarrow\mathbb{R}^{+}$
and a $k$-query tester that accepts members of $\calf_{(C_{k},\Sigma)}$
with probability $1$ and rejects $f$ that are $\eps$-far from $\calf_{(C_{k},\Sigma)}$
with probability at least $\tau(\eps)$. \end{theorem}

However, in strong contrast to Theorem \ref{thm:main2}, we show that
unless $\Sigma$ equals $0^{k}$ or $1^{k}$, the class of $(C_{k},\Sigma)$-freeness
properties is not at all very rich semantically.

\begin{theorem} \label{thm:nonmonchar} The class of properties $\{\calf_{(C_{k},\Sigma)}:k\geq3,\Sigma\neq0^{k},\Sigma\neq1^{k}\}$
is only finitely large. \end{theorem}

The goal of Theorem \ref{thm:nonmontest} is not to introduce new
testable properties but rather to illustrate possible techniques for
analyzing local tests that may lead to more classes of testable non-monotone
properties.

\subsection{Overview of proofs}

We now give an outline of the proofs of our main theorems (Theorems~\ref{thm:main1}~and~\ref{thm:main2}),
and also the extensions (Theorems~\ref{thm:nonmontest}~and~\ref{thm:nonmonchar}).

Our claim in Theorem \ref{thm:main1}, that graphic matroid freeness
properties are locally testable, is based on analyzing the structure
of dependencies among elements of a graphic matroid. To this end,
we first recall the classification of linear forms due to Green and
Tao in \cite{GreenTao:LinearPrimes}. We require a minor reformulation
of their definition since, for us, the structure of the linear constraints
is described by elements of a matroid.

\begin{defn} Given a binary matroid $\calm$ represented by $v_{1},\dots,v_{k}\in\cube^{k}$,
we say that $\calm$ has {\em complexity $c$ at coordinate $i$}
if we can partition $\{v_{j}\}_{j\in[k]\backslash\{i\}}$ into $c+1$
classes such that $v_{i}$ is not in the span of any of the classes.
We say that ${\cal {M}}$ has {\em complexity $c$} if $c$ is
the minimum such that $\calm$ has complexity $c$ at coordinate $i$
for all $i\in[k]$. \end{defn}

The above definition makes sense because the span of a set of elements
is not dependent on the specific basis chosen to represent the matroid.
As a motivating example, consider the graphic matroid of $C_{k}$
studied by Green \cite{Gre03}. It can be represented by $v_{1}={e}_{1},v_{2}={e}_{2},\dots,v_{k-1}={e}_{k-1}$
and $v_{k}={e}_{1}+\cdots+{e}_{k-1}$. We see then that the graphic
matroid of $C_{k}$ has complexity $1$ because for every $i<k$,
the rest of the matroid elements can be partitioned into two sets
$\{{e}_{j}\}_{j\neq i}$ and $\left\{ \sum_{j\in[k]}{e}_{j}\right\} $
such that $v_{i}$ is not contained in the span of either set, and
for $i=k$, any nontrivial partition of the remaining elements ensures
that $v_{k}$ does not lie in the span of either partition. In Section
\ref{sec:Graphic-is-1}, we extend this observation about $C_{k}$
to all graphs. \begin{lem}\label{lem:graphicisone} For all graphs
$G$, the graphic matroid of $G$ has complexity $1$. \end{lem}

Green and Tao \cite{GreenTao:LinearPrimes} showed that if a matroid
$\mathcal{M}$ has complexity $c$ and if $A$ is a subset of $\cube^{n}$,
then the number of linear maps $L:\cube^{k}\to\cube^{n}$ such that
$L(v_{i})\in A$ for all $i\in[k]$ is controlled by the $(c+1)$--th
Gowers uniformity norm of $A$. Previously, Green proved \cite{Gre03}
an arithmetic regularity lemma, which essentially states that any
set $A\subseteq\cube^{n}$ can be partitioned into subsets of affine
subspaces such that nearly every partition is nearly uniform with
respect to linear tests. We show in Section \ref{sec:Matroids-testable}
how to combine these two results to obtain the following:

\begin{lem}\label{lem:oneistestable} Given any binary matroid $\calm$
represented by $v_{1},\dots,v_{k}\in\cube^{k}$, if $\calm$ has complexity
$1$, then there exists a function $\tau:\mathbb{R}^{+}\rightarrow\mathbb{R}^{+}$
and a $k$-query tester that accepts members of $\calf_{\calm}$ with
probability $1$ and rejects $f$ that are $\eps$-far from $\calf_{\calm}$
with probability at least $\tau(\eps)$. \end{lem}

Theorem \ref{thm:main1} directly follows from combining Lemma \ref{lem:graphicisone}
and Lemma \ref{lem:oneistestable}. In fact, Lemma \ref{lem:oneistestable}
implies testability of all matroids that have complexity one, not
only those that are graphic. In Section \ref{sec:Graphic-is-1}, we
give examples of binary matroids that have complexity $1$ and yet
are provably not graphic.

Theorem \ref{thm:main2} provides a proper hierarchy among the graphical
properties. Moreover, the containments $\mathcal{P}_{1}\subsetneq\mathcal{P}_{2}$
in this hierarchy are shown to be {}``statistically proper'' in
the sense that we demonstrate functions $f$ that are $\eps$-far
from $\mathcal{P}_{1}$ but are in $\mathcal{P}_{2}$. The theorem
implies the following hierarchy: \[
\cdots\subsetneq C_{k+2}\text{-free}\subsetneq C_{k}\text{-free}\subsetneq\cdots\subsetneq C_{3}\text{-free}=K_{3}\text{-free}\subsetneq\cdots\subsetneq K_{k}\text{-free}\subsetneq K_{\binom{k}{2}+2}\text{-free}\subsetneq\cdots\]
 Thus, the class of properties $\calf_{G}$ does indeed contain infinitely
many more properties than the cycle freeness properties considered
by Green \cite{Gre03}.

Both the hierarchy among the cyclic freeness properties and among
the clique freeness properties are derived in Section \ref{sec:Infinitely-properties}
using a general technique. In order to show a statistically proper
containment $\calm_{1}$-free $\subsetneq\calm_{2}$-free, we construct
a function $f$ that, by its definition, contains $\mathcal{M}_{1}$
at a large number of linear maps and so is far from being $\mathcal{M}_{1}$-free.
On the other hand, the construction ensures that if $f$ is also not
$\mathcal{M}_{2}$-free, then there is a {\em matroid homomorphism}
from $\calm_{2}$ to $\calm_{1}$. We define a matroid homomorphism
from a binary matroid $\calm_{2}$ to a binary matroid $\calm_{1}$
to be a map from the ground set of $\calm_{2}$ to the ground set
of $\calm_{1}$ which maps cycles to cycles. The separation between
$\mathcal{M}_{2}$-freeness and $\mathcal{M}_{1}$-freeness is then
obtained by proving that there do not exist any matroid homomorphisms
from $\calm_{2}$ to $\calm_{1}$. This proof framework suffices for
both the claims in Theorem \ref{thm:main2} and is reminiscent of
proof techniques involving graph homomorphisms in the area of graph
property testing (see \cite{AlonShapiraSurvey} for a survey).

Theorem \ref{thm:nonmontest} is the result of a more involved application
of the regularity lemma. To deal with non-monotone properties, we
employ a different {}``rounding'' scheme inspired by the testability
of non-monotone graph properties in \cite{AFNS}. Unlike Szemerédi's
regularity lemma, a {}``strong form'' of the arithmetic regularity
lemma is not known, so we restrict our attention to cyclic matroids
and exploit the additive structure of the pattern. Theorem \ref{thm:nonmonchar}
is based on a characterization theorem in Appendix \ref{sub:Characterization}
that classifies $(C_{k},\Sigma)$-freeness properties into $9$ classes
when $\Sigma\neq0^{k},1^{k}$.

\section{Freeness of complexity $1$ matroids is testable\label{sec:Matroids-testable}}

In this section we prove Lemma \ref{lem:oneistestable}. Before doing
so, we fix our notation and provide a quick background on Fourier
analysis. If $H$ is a subgroup of $G$, the cosets of $H$ are indicated
by $g+H$, with $g$ in $G$. Let $f_{g+H}:H\rightarrow\cube$ denote
$f$ restricted to the coset $g+H$, defined by sending $h$ to $f(g+h)$;
that is, for every $h\in H,g\in G$, $f_{g+H}(h):=f(g+h)$. For $\sigma\in\cube$,
we define $\mu_{\sigma}(f_{g+H}):=\Pr_{h\in H}[f_{g+H}(h)=\sigma]$
to be the density of $\sigma$ in $f$ restricted to coset $g+H$.

\subsection{Fourier analysis and Green's regularity lemma\label{sub:Fourier-Analysis}}

\begin{defn}[Fourier transform] If $f:\cube^{n}\rightarrow\cube$,
then we define its Fourier transform $\widehat{f}:\cube^{n}\rightarrow\mathbb{R}$
to be $\widehat{f}(\alpha)=\Exp_{x\in\cube^{n}}f(x)\chi_{\alpha}(x)$,
where $\chi_{\alpha}(x)=(-1)^{\sum_{i\in[n]}\alpha_{i}x_{i}}$. $\widehat{f}(\alpha)$
is called the \emph{Fourier coefficient} of $f$ at $\alpha$, and
the $\{\chi_{\alpha}\}_{\alpha}$ are the characters of $\cube^{n}$.
\end{defn} 

It is easy to see that for $\alpha,\beta\in\cube^{n}$, $\Exp\chi_{\alpha}\cdot\chi_{\beta}$
is 1 if $\alpha=\beta$ and $0$ otherwise. Since there are $2^{n}$
characters, the characters form an orthonormal basis for functions
on $\cube^{n}$, and we have the Fourier inversion formula \[
f(x)=\sum_{\alpha\in\cube^{n}}\widehat{f}(\alpha)\chi_{\alpha}(x)\]
 and Parseval's Identity \[
\sum_{\alpha\in\cube^{n}}\widehat{f}(\alpha)^{2}=\Exp_{x}\left[f(x)^{2}\right].\]

Next we turn to Green's arithmetic regularity lemma, the crux of the
analysis of our local testing algorithm. Green's regularity lemma
over $\cube^{n}$ is a structural theorem for Boolean functions. It
asserts that for every Boolean function, there is some decomposition
of the Hamming cube into cosets, such that the function restricted
to most of these cosets are uniform and pseudorandom with respect
to the linear functions. An alternate and equivalent way is that no
matter where we cut the Boolean cube by a hyperplane, the densities
of $f$ on the two halves of the cube separated by the hyperplane
do not differ greatly. Formally, we say that a function is uniform
if all of its nonzero Fourier coefficients are small.

\begin{defn}[Uniformity] For every $0<\eps<1$, we say that a function
$f:\ \cube^{n}\rightarrow\cube$ is \emph{$\eps$-uniform} if for
every $\alpha\neq0\in\cube^{n},\left|\widehat{f}(\alpha)\right|\leq\eps$.
\end{defn}

Recall that we let $W(t)$ denote a tower of twos with height $\left\lceil t\right\rceil $.
To obtain a partition of the Hamming cube that satisfies the required
uniformity requirement, the number of cosets in the partition may
be rather large. More precisely,

\begin{lem}[Green's Regularity Lemma]\label{lemma:Green's regularity}Let
$f:\ \cube^{n}\rightarrow\cube$. For every $0<\eps<1$, there exists
a subspace $H$ of $G=\cube^{n}$ of co-dimension at most $W(\eps^{-3})$,
such that $\Pr_{g\in G}\left[f_{g+H}\text{ is }\eps\text{-uniform}\right]\geq1-\eps$.
\end{lem}

\subsection{Testability of complexity $1$ matroid freeness}

The proposition below is proved in \cite{GreenTao:LinearPrimes}.
Collectively, statements capturing the phenomenon that expectation
over certain forms are controlled by varying degrees of the Gowers
norm are termed \emph{generalized von-Neumann type Theorems} in the
additive combinatorics literature. In particular, as we only require
the degree $2$ Gowers norm of a function, which is equivalent to
the $\ell_{4}$ norm of the function's Fourier transform. The version
we state here requires the functions $f_{i}$ to be over $\cube^{n}$
and possibly distinct; however as explained by Gowers and Wolf~\cite{GowersWolf},
both conditions can be easily satisfied. 

\begin{prop}[implicit in \cite{GreenTao:LinearPrimes}]\label{prop:von Neumann}
Suppose a binary matroid $\calm=\{v_{1}\ldots,v_{k}\}$ has complexity
$1$ and let $f_{1},\ldots,f_{k}:\cube^{n}\rightarrow\cube$. Then
\[
\Exp_{L:\cube^{k}\rightarrow\cube^{n}}\thinspace\left[\prod_{i=1}^{k}\thinspace f_{i}(L(v_{i}))\right]\leq\min_{i\in[k]}\thinspace\left(\sum_{\alpha\in\cube^{n}}\widehat{f}_{i}(\alpha)^{4}\right)^{1/4}.\]
 \end{prop} It is an easy deduction from Proposition \ref{prop:von Neumann}
to see that if $f$ is uniform, then the number of linear maps $L$
where $f$ has a $\calm$-pattern is close to $\Exp[f]^{m}N^{d}$,
where $N=2^{n}$. Combining this observation with the regularity lemma,
we prove Lemma \ref{lem:oneistestable}.

\begin{proof}[Proof of Lemma \ref{lem:oneistestable}] Consider a
test that picks a linear map $L$ uniformly at random from all linear
maps from $\cube^{k}\rightarrow\cube^{n}$ and rejects iff for all
$i\in[k]$, $f(L(v_{i}))=1$. Clearly the test has completeness one.

Now we analyze the soundness of this test. Suppose $f$ is $\eps$-far
from being $\calm$-free. We want to show that the test rejects with
probability at least $\tau(\eps)$, such that $\tau(\eps)>0$ whenever
$\tau>0$. Let $a(\eps)$ and $b(\eps)$ be two functions of $\eps$
that satisfy the constraint $a(\eps)+b(\eps)<\eps$, we shall specify
these two functions at the end of the proof. We now apply Lemma~\ref{lemma:Green's regularity}
to $f$ to obtain a subspace $H$ of $G$ of co-dimension at most
$W(a(\eps)^{-3})$. Consequently, $f$ restricted to all but at most
$a(\eps)$ fraction of the cosets of $H$ are $a(\eps)$-uniform.
We define a reduced function $f^{R}:\cube^{n}\rightarrow\cube$ as
follows.

For each $g\in G$, if $f$ restricted to the coset $g+H$ is $a(\eps)$-uniform,
then define\[
f_{g+H}^{R}(x)=\begin{cases}
0 & \text{ if }\mu(f_{g+H})\leq b(\eps)\\
f_{g+H} & \mbox{ otherwise. }\end{cases}\]

Else, define $f_{g+H}^{R}=0$.

Note that at most $a(\eps)+b(\eps)$ fraction of modification has
been made to $f$ to obtain $f^{R}$. Since $f$ is $\eps$-far from
being $\calm$-free, $f^{R}$ has a $\calm$-pattern at some linear
map $L$. More precisely, for every $i\in[k]$, $f^{R}(L(v_{i}))=1$.
Now consider the cosets $L(v_{i})+H$. By our construction of $f^{R}$,
we know that $f$ restricted to each of these cosets is $a(\eps)$-uniform
and at least $b(\eps)$ dense. We will count the number of linear
maps $\phi:\cube^{k}\rightarrow H$ such that $f$ has a $\calm$
pattern at $L+\phi$. Notice that the probability the test rejects
is at least \[
2^{-k\cdot W(a(\eps)^{-3})}\thinspace\Pr_{\phi:\cube^{k}\rightarrow H}\left[\forall i,f_{L(v_{i})+H}(\phi(v_{i}))=1\right].\]
 To lower-bound this rejection probability, it suffices to show that
the probability \[
\Pr_{\phi:\cube^{k}\rightarrow H}\left[\forall i,f_{L(v_{i})+H}(\phi(v_{i}))=1\right]\]
 is bounded below by at least some constant depending on $\eps.$
To this end, we rewrite this probability as \begin{equation}
\Exp_{\phi:\cube^{k}\rightarrow H}\left[\prod_{i\in[k]}f_{i}(\phi(v_{i}))\right],\label{eq:reject expectation}\end{equation}
 where $f_{i}=f_{L(v_{i})+H}$. By replacing each function $f_{i}$
by $\Exp f_{i}+(f_{i}-\Exp f_{i})$, it is easy to see that the above
expression can be expanded into a sum of $2^{k}$ terms, one of which
is $\prod_{i\in[k]}\Exp f_{i}$, which is at least $b(\eps)^{k}$.
For the other $2^{k}-1$ terms, by applying Proposition~\ref{prop:von Neumann}
and using Parseval's Identity, each of these terms is bounded above
by $a(\eps)^{1/2}$. So Equation~\ref{eq:reject expectation} is
at least $b(\eps)^{k}-(2^{k}-1)a(\eps)^{1/2}$. To finish the analysis,
we need to specify $a(\eps),b(\eps)$ such that $b(\eps)^{k}-(2^{k}-1)a(\eps)^{1/2}>0$
and $a(\eps)+b(\eps)<\eps$. Both are satisfied by setting $a(\eps)=(\frac{\eps}{4})^{2k},\, b(\eps)=\frac{\eps}{2}$.
Thus, the rejection probability is at least $\tau(\eps)\geq2^{-k(\, W((4/\eps)^{6k})+2)}\cdot\eps^{k}$,
completing the proof. \end{proof}

\section{Graphic matroids have complexity $1$\label{sec:Graphic-is-1}}

Here we prove that graphic matroids have complexity 1. While the proof
is simple, we believe it sheds insight into the notion of complexity
and shows that even the class of 1-complexity matroids is quite rich.

\begin{proof}[Proof of Lemma \ref{lem:graphicisone}] Recall that
throughout we are assuming $G$ to be a simple graph. Fix an arbitrary
edge $e$ in $G$ with vertices $v_{1}$ and $v_{2}$ as its two ends.
We partition the remaining edges of $G$ into two sets $S_{1}$ and
$S_{2}$ such that, if an edge is incident to $v_{1}$ then it is
in $S_{1}$ and otherwise, it is in $S_{2}$. Because $G$ is simple,
a cycle in $G$ containing $e$ must include an edge (apart from edge
$e$) which is incident to $v_{1}$ and another edge (other than $e$)
which is not incident to $v_{1}$. Therefore $e$ is not in the span
of either $S_{1}$ or $S_{2}$. \end{proof}

As we have seen earlier, Lemma \ref{lem:oneistestable} holds for
any matroid of complexity $1$. Hence, it is a natural question to
ask whether there exist non-graphic matroids which have complexity
$1$. In Appendix~\ref{app:nongraphic} we show that such matroids
do exist. It is an open question to come up with a natural characterization
of matroids having complexity $1$.

\section{Infinitely many monotone properties\label{sec:Infinitely-properties}}

In this section we prove Theorem~\ref{thm:main2}, that there are
infinitely many matroids for which the property of being $\calm$-free
are pairwise very different.

To do so we consider a pair of target matroids $\calm_{1}$ and $\calm_{2}$.
Based on just the first matroid $\calm_{1}$, we create a canonical
function $f=f_{\calm_{1}}:\cube^{n}\to\cube$. We show, using a simple
analysis, that this canonical function is far from being $\calm_{1}$
free. We then show that if this function has an instance of $\calm_{2}$
inside, then there is a {}``homomorphism'' (in a sense we define
below) from $\calm_{2}$ to $\calm_{1}$. Finally we show two different
ways in which one can rule out homomorphisms between pairs of graphic
matroids; one based on the odd girth of the matroids, and the other
based on the maximum degree of $\calm_{1}$. Together these ideas
lead to proofs of distinguishability of many different matroids.

\begin{defn} Given a binary matroid $\calm$ represented by vectors
$v_{1},\ldots,v_{k}\in\cube^{k}$, and integer $n\geq k$, let the
canonical function $f=f_{\calm}:\cube^{n}\to\cube$ be given by $f(x,y)=1$
if $x\in\{v_{1},\ldots,v_{k}\}$ and $0$ otherwise; where $x\in\cube^{k}$
and $y\in\cube^{n-k}$. \end{defn}

\begin{claim} \label{claim:e-far}Let $\calm$ be a binary matroid
with $v_{i}\ne0$ for all $i\in\{1,\ldots,k\}$. Then $f_{\calm}$
is $\frac{1}{2^{k}}$--far from being $\calm$-free. \end{claim}

\begin{proof} Note that if we consider the linear map $L:\cube^{k}\to\cube^{n}$
that sends $x$ to $\langle x,0\rangle$, then $f$ contains $\calm$
at $L$. So $f$ is not $\calm$-free. However we wish to show that
any function that is $2^{-k}$-close to $f$ contains $\calm$ somewhere.
Fix a function $g$ such that $\delta(f,g)=\delta<2^{-k}$. We will
show that $g$ contains $\calm$ somewhere.

For $i\in[k]$ let $\delta_{i}=\Pr_{y\in\cube^{n-k}}[f(v_{i},y)\ne g(v_{i},y)]$.
Note that $\sum_{i=1}^{k}\delta_{i}\leq2^{k}\cdot\delta<1$. Now consider
a random linear map $L_{1}:\cube^{k}\to\cube^{n-k}$, and its extension
$L:\cube^{k}\to\cube^{n}$ given by $L(x)=\langle x,L_{1}(x)\rangle$.
For every non-zero $x$ and in particular for $x\in\{v_{1},\ldots,v_{k}\}$,
we have $L_{1}(x)$ is distributed uniformly over $\cube^{n-k}$.
Thus, for any fixed $i\in[k]$, we have $\Pr_{L_{1}}[g(L(v_{i}))\ne1]\leq\delta_{i}$.
By the union bound, we get that $\Pr_{L_{1}}[\exists i{\rm {\, s.t.\,}}g(L(v_{i}))\ne1]\leq\sum_{i}\delta_{i}<1$.
In other words there exists a linear map $L_{1}$ (and thus $L$)
such that for every $i$, $g(L(v_{i}))=1$ and so $g$ contains $\calm$
at $L$. \end{proof}

We now introduce our notion of a {}``homomorphism'' between binary
matroids. (We stress that the phrase homomorphism is conjured up here
and we are not aware of either this notion, or the phrase being used
in the literature. We apologize for confusion if this phrase is used
to mean something else.)

\begin{defn} Let $\calm_{1}$ and $\calm_{2}$ be binary matroids
given by $v_{1},\ldots,v_{k}\in\cube^{k}$ and $w_{1},\ldots,w_{\ell}\in\cube^{\ell}$.
We say that $\calm_{2}$ has a homomorphism to $\calm_{1}$ if there
is a map $\phi:\{w_{1},\ldots,w_{\ell}\}\to\{v_{1},\ldots,v_{k}\}$
such that for every set $T\subseteq[\ell]$ such that $\sum_{i\in T}w_{i}=0$,
it is the case that $\sum_{i\in T}\phi(w_{i})=0$. \end{defn}

For graphic matroids, the matroid-homomorphism from $G$ to $H$ is
a map from the edges of $G$ to the edges of $H$ that ensures that
cycles are mapped to even degree subgraphs of $H$.

\begin{lem} \label{lem:homomorphism}If the canonical function $f_{\calm_{1}}$
contains an instance of $\calm_{2}$ somewhere, then $\calm_{2}$
has a homomorphism to $\calm_{1}$. \end{lem}

\begin{proof} Let $f=f_{\calm_{1}}$ contain $\calm_{2}$ at $L$.
So $L:\cube^{\ell}\to\cube^{n}$ is a linear map satisfying $f(L(w_{i}))=1$
for every $i\in[\ell]$. Now consider the projection map $\pi:\cube^{n}\to\cube^{k}$
which sends $\langle x,y\rangle$ to $x$ (where $x\in\cube^{k}$
and $y\in\cube^{n-k}$).

We claim that the map $\phi$ which sends $x$ to $\pi(L(x))$ gives
a homomorphism from $\calm_{2}$ to $\calm_{1}$. On the one hand
$\phi$ is linear and so if $\sum_{i\in T}w_{i}=0$, then we have
$\sum_{i\in T}\phi(w_{i})=\phi(\sum_{i\in T}w_{i})=\phi(0)=0$. On
the other hand, we also have that $\phi(w_{i})\in\{v_{1},\ldots,v_{k}\}$.
This is true since $f(L(w_{i}))=1$, which implies, by the definition
of the canonical function $f$ that $\pi(L(w_{i}))\in\{v_{1},\ldots,v_{k}\}$.
Thus $\phi$ satisfies the requirements of a homomorphism from $\calm_{2}$
to $\calm_{1}$. \end{proof}

The above lemma now motivates the search for matroids $\calm_{2}$
that are not homomorphic to $\calm_{1}$. Proving non-homomorphism
in general may be hard, but we give a couple of settings where we
can find simple proofs. Each addresses a different case of Theorem~\ref{thm:main2}.

\global\long\def\og{\mathop{\rm og}}
 For a matroid $\calm$, let its {\em odd girth}, denoted $\og(\calm)$,
be the size of the smallest dependent set of odd cardinality, i.e.
the size of the smallest odd set $T\subseteq[\ell]$ such that $\sum_{i\in T}w_{i}=0$.

\begin{lem}\label{lem:odd-girth} If $\calm_{2}$ has a homomorphism
to $\calm_{1}$, then $\og(\calm_{2})\geq\og(\calm_{1})$. \end{lem}

\begin{proof} Let $\phi$ be a homomorphism from $\calm_{2}$ to
$\calm_{1}$ and let $T\subseteq[\ell]$ denote the smallest odd dependent
set of $\calm_{2}$. Now let $T'\subseteq[k]$ be the set $T'=\{j\in[k]|\#\{i\in T|\phi(w_{i})=v_{j}\}{\rm {~is~odd}\}}$.
On the one hand, we have $T'$ has odd cardinality; and on the other,
we have $0=\sum_{i\in T}\phi(w_{i})=\sum_{j\in T'}v_{j}$. So $T'$
is an odd dependent set in $\calm_{1}$. The lemma follows since $|T|\geq|T'|$.
\end{proof}

For graphic matroids constructed from the odd cycle graph $C_{k}$,
we have that its odd girth is just $k$ and so the above lemmas combine
to give that $C_{k}$-freeness is distinguishable from $C_{k+2}$-freeness,
and this suffices to prove Part (1) of Theorem~\ref{thm:main2}.

However the odd girth criterion might suggest that $G$-freeness for
any graph containing a triangle might be equivalent. Below we rule
this possibility out.

\begin{lem}\label{lem:K_n} Let $\calm_{1}$ be the graphic matroid
of the complete graph $K_{a}$ on $a$ vertices, and let $\calm_{2}$
be the graphic matroid of $K_{b}$. Then, if $b\geq\binom{a}{2}+2$,
there is no homomorphism from $\calm_{2}$ to $\calm_{1}$. \end{lem}

\begin{proof} Assume otherwise and let $\phi$ be such a homomorphism.
Fix any vertex of $K_{b}$ and let $e_{1},\ldots,e_{b-1}$ denote
the $b-1$ edges incident to this vertex. By the pigeonhole principle,
(since $b-1>\binom{a}{2}$) there must exist a pair of incident edges
$e_{i}$ and $e_{j}$ such that $\phi(e_{i})=\phi(e_{j})$. But now
let $f$ denote the edge which forms a triangle with $e_{i}$ and
$e_{j}$. Since in $K_{b}$ we have $e_{i}+e_{j}+f=0$ (viewing these
elements as vectors over $\cube$), it must be that $\phi(f)=\phi(e_{i})+\phi(e_{j})=0$
which is not an element of the ground set of $\calm_{1}$. This yields
the desired contradiction. \end{proof}

We are now ready to prove Theorem~\ref{thm:main2}.

\begin{proof}[Proof of Theorem~\ref{thm:main2}] First note that $C_{k+2}$-free
functions are also $C_{k}$-free. Informally, suppose a function $f$
has a $k$ cycle at point $x_{1},\ldots,x_{k}$, i.e., $f(x_{i})=1$
at these points and $\sum_{i}x_{i}=0$. Then $f$ has a $k+2$ cycle
at the points $x_{1},x_{1},x_{1},x_{2},\ldots,x_{k}$. (This informal
argument can obviously be converted to a formal one once we specify
the graphic matroids corresponding to $C_{k}$ and $C_{k+2}$ formally.)

On the other hand, if we take $\calm_{1}$ to be the graphic matroid
corresponding to $C_{k}$ and $f$ to be the canonical function corresponding
to $\calm_{1}$, then by Claim~\ref{claim:e-far} it is $2^{-k}$-far
from $\calm_{1}$-free, and by Lemmas~\ref{lem:homomorphism} and
\ref{lem:odd-girth} it does not contain $\calm_{2}$, the graphic
matroid of $C_{k+2}$.

For the second part of the theorem, note that every property that
is $G$-free is also $H$-free if $G$ is a subgraph of $H$. Thus
$K_{k}$-free is contained in $K_{\ell}$ free if $k\leq\ell$. The
proper containment can now be shown as above, now using Claim~\ref{claim:e-far}
and Lemmas~\ref{lem:homomorphism} and \ref{lem:K_n}. \end{proof}

\section{Conclusions and future work\label{sec:Conclusions}}

We introduced an infinite family of properties of Boolean functions
and showed them to be testable. Unfortunately, we were only able to
analyze the tests when the matroid $\calm$ was graphic and the pattern
was monochromatic. This raises a plethora of new problems that we
describe below.

The first natural quest is to generalize the problem to the solution
to the case when the matroid is arbitrary over $\cube$, and further
to the case when the matroid is over other fields. We note that this
seems to pose significant technical hurdles and indeed even the simple
property of being free of the matroid $\{e_{1},e_{2},e_{3},e_{1}+e_{2},e_{2}+e_{3},e_{3}+e_{1},e_{1}+e_{2}+e_{3}\}$
(where $e_{1},e_{2},e_{3}$ are linearly independent vectors) remains
open.

Next, it would be nice to extend the results to the case where the
pattern $\Sigma$ is an arbitrary binary string, as opposed to being
monotone. We did manage to extend this in the special case where $\calm$
is a cyclic matroid, but in this case the extension is not very interesting.
We do feel that our proof techniques already capture some non-trivial
other cases, but are far from capturing all cases, even for graphic
matroids.

Extending the patterns further, there is no real reason to view the
range as a field element, so a major generalization would be to consider
matroids over arbitrary fields, and letting the range be some arbitrary
finite set $R$ where the forbidden pattern $\Sigma\in R^{k}$. (We
don't believe there should be any major technical barriers in this
step, once we are able to handle arbitrary 0/1 patterns $\Sigma$.)
Finally, all the above problems consider the case of a single forbidden
pattern (and its linear transformations).

These properties were specified by a matroid $\calm$ on $k$ elements
and a pattern $\Sigma\subseteq\{0,1\}^{k}$. However to capture the
full range of linear-invariant non-linear properties that allow one-sided
error local tests, we should also allow the conjunction of a constant
number of constraints. We believe this could lead to a characterization
of all linear-invariant non-linear properties that allow one-sided
error local tests.

In a different direction, we feel that it would also be nice to develop
richer techniques to show the distinguishability of syntactically
different properties. For instance, even for the graphic case we don't
have a good understanding of when two different graphs represent essentially
the same properties, and when they are very different.

\section*{Acknowledgments}

We are grateful to Kevin Matulef for suggesting this research direction.
We thank Tali Kaufman and Swastik Kopparty for helpful discussions.
We thank Asaf Shapira for drawing our attention to his preprint \cite{Shapira}.

\bibliographystyle{plain}
\bibliography{main}

\appendix
\section{Matroids\label{sec:def-matroid}}

For background on matroid theory, we refer the reader to~\cite{Wel76}.

\begin{defn} A matroid ${\cal {M}}$ is a finite set $S$ (called
\emph{ground set}) and a collection $\mathcal{F}$ of subsets of $S$
(called \emph{independent sets}) such that the following hold:
\begin{enumerate}
\item $\emptyset\in\mathcal{F}$. 
\item If $X\in\mathcal{F}$ and $Y\subseteq X$, then $Y\in\mathcal{F}$. 
\item If $X$ and $Y$ are both in $\mathcal{F}$ with $|X|=|Y|+1$, then
there exists an $x\in X\setminus Y$ such that $Y\cup x\in\mathcal{F}$. 
\end{enumerate}
\end{defn}

A matroid $\calm$ on a ground set $S=\{x_{1},\ldots,x_{k}\}$ is
said to be {\em linear} if there exists a field $\mathbb{F}$ and
vectors $v_{1},\ldots,v_{k}\in\mathbb{F}^{k}$ such that some subset
$\{x_{i}|i\in T\}$ indexed by $T\subseteq\{1,\ldots,k\}$ is independent
if and only if the corresponding vectors $\{v_{i}|i\in T\}$ are linearly
independent. A linear matroid is {\em binary} if $\mathbb{F}=\cube$.

\section{Non-graphic matroids of complexity $1$}

\label{app:nongraphic}

First, we make the following claim that follows immediately from the
definition of cographic matroids and the notion of complexity. \begin{claim}\label{clm:cog}
A cographic matroid $\calm^{*}(G)$ has complexity $1$ if and only
if, for every edge $e\in E(G)$, there is a partition of $E(G)\setminus\{e\}$
into two disjoint sets $A$ and $B$ such that both of the subgraphs
$(V(G),A)$ and $(V(G),B)$ are connected. \end{claim}

\begin{prop} There is a matroid with complexity one that is not graphic.
\end{prop} \begin{proof} Consider the cographic matroid of $K_{5}$.
Embed $K_{5}$ in the plane as a pentagon and all its diagonals. Fix
an outer edge $e$ and partition the remaining $9$ edges into two
sets. One is the $4$ outer edges and the other is the remaining $5$
diagonal edges. Clearly both outer-edge set and diagonal-edge set
make the five vertices connected. Therefore by Claim \ref{clm:cog},
the cographic matroid of $K_{5}$ is of complexity one. On the other
hand, by a theorem of Tutte \cite{Tutte}, a matroid cannot be graphic
if it contains $\calm^{*}(K_{5})$ as a minor, which $\calm^{*}(K_{5})$
clearly does. So, $\calm^{*}(K_{5})$ is an example of a non-graphic
matroid that has complexity $1$. \end{proof}

We remark that not all cographic matroids have complexity $1$. For
example, the cographic matroid of $K_{3,3}$ cannot have complexity
$1$ because if we remove an edge from $K_{3,3}$, there do not remain
enough edges to form two edge-disjoint connected graphs on $6$ vertices,
violating Claim \ref{clm:cog}.

\section{Testing non-monotone properties\label{sec:Non-monotone-properties}}

In this section we prove Theorem~\ref{thm:nonmontest}. (Readers
may find it useful to recall the background material in Section \ref{sub:Fourier-Analysis}.)
We show that for non-monotone properties, i.e., when $\Sigma\neq0^{k}$
or $1^{k}$, the property of $(\calm,\Sigma)$-free is testable when
the underlying graph is a cycle. However, as opposed to Section \ref{sec:Infinitely-properties},
the number of non-monotone properties associated with cycles is finite.
In fact we give a complete characterization of these non-monotone
properties in Appendix~\ref{sub:Characterization}.

\begin{proof}[Proof of Theorem~\ref{thm:nonmontest}] Suppose we have
oracle access to a function $f:\cube^{n}\rightarrow\cube$. Consider
the following $k$-query test $T$, which selects a linear map $L:\cube^{k}\rightarrow\cube^{n}$
uniformly at random from all such possible linear maps. $T$ has oracle
access to $f$ and queries $f$ at the points $L(v_{1}),\ldots,L(v_{k})$.
$T$ rejects iff all of these points are evaluated to $1$. If $f$
is $(\calm,\Sigma)$-free, $T$ never rejects and has completeness
$1$.

Now we analyze the soundness of $T$. Suppose that $f$ is $\eps$-far
from being $(\calm,\Sigma)$-free. We want to show that $T$ rejects
with probability at least $\tau(\eps)$, such that $\tau(\eps)>0$
whenever $\eps>0$.

Let $\frac{1}{2}<\eta<1$ be any constant, and $a(\eps)$ and $b(\eps)$
be functions of epsilons that satisfy the constraints $a(\eps)+b(\eps)<\eps$
and $1-\eta>b(\eps)$. We shall specify these two functions at the
end of the proof.

Now let $G$ denote $\cube^{n}$. We apply Lemma~\ref{lemma:Green's regularity}
to $f$ to obtain a subspace $H$ of $G$ of co-dimension at most
$W(a(\eps)^{-3})$. We define a reduced function $f^{R}:\cube^{n}\rightarrow\cube$
as follows. We assume that $\Sigma$ has at least two occurrences
of $1$. (Otherwise it has at least two occurrences of $0$, and in
the construction of $f^{R}$, we flip the roles of $1$ and $0$ when
$f_{g+H}$ is not uniform. The rest of the proof will proceed analogously,
and we leave its verification to the readers.)

For each $g\in G$, if $f$ restricted to the coset $g+H$ is $a(\eps)$-uniform,
then define\[
f_{g+H}^{R}=\begin{cases}
0 & \text{ if }\mu(f_{g+H})<b(\eps)\\
1 & \text{ if }\mu(f_{g+H})>1-b(\eps)\\
f_{g+H} & \mbox{ otherwise. }\end{cases}\]

Else, define \[
f_{g+H}^{R}=\begin{cases}
1 & \text{ if }\mu(f_{g+H})\geq\eta\\
0 & \mbox{ otherwise. }\end{cases}\]

Note that at most $a(\eps)+b(\eps)$ fraction of modification has
been made to $f$ to obtain $f^{R}$, so $f^{R}$ is $\eps$-close
to $f$. By assumption, $f$ is $\eps$-far from $(\calm,\Sigma)$-free,
so $f^{R}$ has a $(\calm,\Sigma)$ pattern at some linear map $L:\cube^{k}\rightarrow\cube^{n}$,
i.e., for each $i\in[k]$, $f^{R}(L(v_{i}))=\sigma_{i}$, where $\Sigma=\left\langle \sigma,\ldots,\sigma_{k}\right\rangle $,
and $\sigma_{i}\in\cube$. Now consider the cosets $L(v_{i})+H$.
By our choice of rounding, $f$ restricted to each $L(v_{i})+H$ is
dense in the symbol $\sigma_{i}$, i.e., $\mu_{\sigma_{i}}(f_{L(v_{i})+H})\geq b(\eps)$.
since $1-\eta\geq b(\eps)$. We want to show that there are many $(\calm,\Sigma)$
patterns spanning across these cosets. In particular, we restrict
our attention to the relative number of $(\calm,\Sigma)$-patterns
at linear maps of the form $L+\phi$ , where $\phi$ maps linearly
from $\cube^{k}$ to $H$. Notice that the probability the test $T$
rejects is at least \[
2^{-(k-1)W(a(\eps)^{-3})}\cdot\Pr_{\phi:\cube^{k}\rightarrow H}[\forall i\in[k],f_{L(v_{i})+H}(\phi(v_{i}))=\sigma_{i}].\]

It suffices to show that the probability \begin{equation}
\Pr_{\phi:\cube^{k}\rightarrow H}[\forall i\in[k],f_{L(v_{i})+H}(\phi(v_{i}))=\sigma_{i}]\label{eq:Test Rej Prob}\end{equation}
 is bounded below by some constant depending only on $\eps$. To this
end, we divide our analysis into two cases, based on whether there
is some $j\in[k]$ such that $f_{L(v_{j})+H}$ is $a(\eps)$-uniform
or not.

\emph{Case $1$}: There exists some $j\in[k]$ such that $f_{L(v_{j})+H}$
is $a(\eps)$-uniform.

For each $i\in[k]$, define $f_{i}:H\rightarrow\cube$ to be $f_{i}=f_{L(v_{i})+H}+\sigma_{i}+1$.
Note that by definition, $\Exp f_{i}\geq b(\eps)$. We begin by arithmetize
Equation \ref{eq:Test Rej Prob} as \[
\Exp_{\phi:\cube^{k}\rightarrow H}\left[\prod_{i\in[k]}f_{i}(\phi(v_{i}))\right].\]

Since $\calm$ is a cyclic matroid, it is not hard to show, by Fourier
expansion, that \[
\Exp_{\phi:\cube^{k}\rightarrow H}\left[\prod_{i\in[k]}f_{i}(\phi(v_{i}))\right]=\sum_{\alpha\in H}\prod_{i\in[k]}\thinspace\widehat{f_{i}}(\alpha).\]

Using the facts that $\Exp f_{i}\geq b(\eps)$, $f_{j}$ is $a(\eps)$-uniform,
there exist two distinct indices $i_{1},i_{2}\neq j\in[k]$ (since
$k\geq3$), Cauchy-Schwarz, and Parseval's Identity, respectively,
we have

\begin{eqnarray*}
\sum_{\alpha\in H}\prod_{i\in[k]}\thinspace\widehat{f_{i}}(\alpha) & \geq & b(\eps)^{k}-\sum_{\alpha\neq0\in H}\thinspace\prod_{i\in[k]}\thinspace\left|\widehat{f_{i}}(\alpha)\right|\\
 & \geq & b(\eps)^{k}-a(\eps)\sum_{\alpha\neq0\in H}\thinspace\prod_{i\in[k]\backslash\{j\}}\thinspace\left|\widehat{f_{i}}(\alpha)\right|\\
 & \geq & b(\eps)^{k}-a(\eps)\sum_{\alpha\neq0\in H}\thinspace\left|\widehat{f_{i_{1}}}(\alpha)\right|\left|\widehat{f_{i_{2}}}(\alpha)\right|,\\
 & \geq & b(\eps)^{k}-a(\eps)\left(\sum_{\alpha\neq0\in H}\widehat{f_{i_{1}}}(\alpha)^{2}\right)^{1/2}\left(\sum_{\alpha\neq0\in H}\widehat{f_{i_{2}}}(\alpha)^{2}\right)^{1/2}\\
 & \geq & b(\eps)^{k}-a(\eps).\end{eqnarray*}

To finish the analysis, we need to specify $a(\eps),b(\eps)$ such
that the constraints $a(\eps)+b(\eps)<\eps$ and $1-\eta>b(\eps)$
are satisfied. Let $b(\eps)=(1-\eta)\cdot\eps$ and $a(\eps)=\frac{1}{2}(1-\eta)^{k}\eps^{k}$,
we have that the rejection probability is at least $\tau(\eps)\geq2^{-(k-1)W(a(\eps)^{-3})}(1-\eta)^{k}\eps^{k}/2$.

\emph{Case} 2: No $j\in[k]$ exists such that $f_{L(v_{j})+H}$ is
$a(\eps)$-uniform.

Since $\calm$ is a cyclic matroid, it is not hard to see that Equation
\ref{eq:Test Rej Prob} is equal to \begin{equation}
\Pr_{x_{1},\ldots,x_{k}\in H;\sum_{i}x_{i}=0}\left[\forall i\in[k],f_{L(v_{i})}(x_{i})=\sigma_{i}\right].\label{eq:Test Rej Prob 2}\end{equation}

Since $\Sigma$ contains at least two occurrences of the symbol $1$,
we may assume without loss of generality that $\sigma_{k-1}=\sigma_{k}=1$.
Fix $x_{1},\ldots,x_{k-2}\in H$ such that $f_{L(v_{i})}(x_{i})=\sigma_{i}$.
Let $z=\sum_{i=1}^{k-2}x_{i}$. Since $\eta>\frac{1}{2}$, by union
bound we have

\begin{eqnarray*}
\Pr_{x\in H}[f_{L(v_{k-1})}(x)=f_{L(v_{k})}(x+z)=1] & = & 1-\Pr_{x\in H}[f_{L(v_{k-1})}(x)=0\mbox{ or }f_{L(v_{k})}(x+z)=0]\\
 & \geq & 1-2(1-\eta)\\
 & > & 0.\end{eqnarray*}

Since for each $i\in[k]$, $f_{L(v_{i})+H}$ is not $a(\eps)$-uniform,
by our choice of rounding, $\Pr_{x\in H}[f_{L(v_{i})}(x_{i})=\sigma_{i}]$
is at least $1-\eta$. By picking $x_{1},\ldots,x_{k-2}$ uniformly
at random from $H$, it is not hard to see that the rejection probability
of the test is at least \[
\tau(\eps)\geq2^{-(k-1)W(a(\eps)^{-3})}(1-\eta)^{k-2}(2\eta-1),\]
 where $a(\eps)=\frac{1}{2}(1-\eta)^{k}\eps^{k}$. \end{proof}

\section{Characterization of cycle free functions\label{sub:Characterization}}

In this section we consider the property of being $(\calm,\Sigma)$-free,
where $\calm$ is the matroid of the $k$-cycle. Syntactically these
appear to be infinitely many different properties. We show that there
are only finitely many distinct properties here when $\Sigma$ is
not equal to $0^{k}$ or $1^{k}$. (As noted in Section~\ref{sec:Infinitely-properties},
when $\Sigma=1^{k}$, we do get infinitely many distinct properties.)

We start with some terminology that describes the distinct families
we get.

\begin{defn}$\mbox{}$$\mbox{\ensuremath{\,}}$
\begin{itemize}
\item Let $\const$ denote the set of constant functions (i.e., the zero
function and the one function). 
\item Let $\mlin$ denote the set of all linear functions, including the
constant functions. (We note that throughout we think of the constant
functions as linear, affine etc.). Let $\blin$ denote the complementary
family, i.e., all functions whose complements are in $\mlin$. 
\item Let $\aff$ denote the set of all affine functions, i.e., the linear
functions and their complements. Let $\baff$ denote the complementary
family. 
\item We use the notation $\flin$ to denote the family of linear subspace
functions and the $0$ function, i.e., $\flin=\{0\}\cup\{f:\cube^{n}\to\cube|f^{-1}(1)$
is a linear subspace of $\cube^{n}\}$. $\bflin$ is the complementary
family. 
\item We use the notation $\faff$ to denote the family of affine subspace
functions and the $0$ function, i.e., $\faff=\{0\}\cup\{f:\cube^{n}\to\cube|f^{-1}(1)$
is an affine subspace of $\cube^{n}\}$. $\bfaff$ is the complementary
family. 
\end{itemize}
\end{defn}

It turns out that for every $k\geq3$ and every $\Sigma\ne0^{k},1^{k}$,
a $(C_{k},\Sigma)$-free family is one of the nine families $\const,\mlin,\blin,\aff,\baff,\flin,\bflin,\faff,\bfaff$.
To give further details, let $Z(\Sigma)$ denote the number of zeroes
in $\Sigma$ and $O(\Sigma)$ denote the number of ones in $\Sigma$.
We have:

\begin{theorem} \label{thm:characterization} For every $k\geq3$
and every $\Sigma\ne0^{k},1^{k}$, a $(C_{k},\Sigma)$-free family
is one of \[
\const,\mlin,\blin,\aff,\baff,\flin,\bflin,\faff,\bfaff.\]
 Specifically:
\begin{enumerate}
\item If $Z(\Sigma)$ and $O(\Sigma)$ are even, the $\calf_{C_{k},\Sigma}=\const$. 
\item If $O(\Sigma)>1$ is odd and $Z(\Sigma)$ is even, then $\calf_{C_{k},\Sigma}=\mlin$.
Complementarily, if $Z(\Sigma)>1$ is odd and $O(\Sigma)$ is even,
then $\calf_{C_{k},\Sigma}=\blin$. 
\item If $O(\Sigma)=1$ and $Z(\Sigma)$ is even, then $\calf_{C_{k},\Sigma}=\flin$.
Complementarily, if $Z(\Sigma)=1$ and $O(\Sigma)$ is even, then
$\calf_{C_{k},\Sigma}=\bflin$. 
\item If $O(\Sigma),Z(\Sigma)>1$ are odd, then $\calf_{C_{k},\Sigma}=\faff$. 
\item If $Z(\Sigma)=1$ and $O(\Sigma)>1$ is odd, then $\calf_{C_{k},\Sigma}=\faff$.
Complementarily, if $O(\Sigma)=1$ and $Z(\Sigma)>1$ is odd, then
$\calf_{C_{k},\Sigma}=\bfaff$. 
\end{enumerate}
\end{theorem} We begin with some simple facts and observations. \begin{fact}[\cite{ParRonSam}]
Let $S$ be an affine subspace. Then $x,y$ and $z$ are all in $S$
implies $x+y+z$ is also in $S$. Conversely, if for any triple $x,y$
and $z$ in $S$ implying $x+y+z$ in $S$, then $S$ is an affine
subspace. \end{fact} 

This fact immediately gives the following observations.

\begin{obs}\label{Obs:Obs 0 char} A function $f:\{0,1\}^{n}\rightarrow\{0,1\}$
is $(C_{4},1110)$-free if and only if $f$ is in $\faff$. \end{obs}

\begin{obs}\label{Obs:Obs 1 char} A function $f:\{0,1\}^{n}\rightarrow\{0,1\}$
is $(C_{3},100)$-free if and only if $f$ is the disjunction (OR)
of linear functions (or the all $1$ function). Consequently, $f$
is $(C_{3},110)$-free if and only if $f$ is in $\flin$. \end{obs} 

\begin{proof} Let $S=\{x\in\{0,1\}^{n}:f(x)=0\}$. If $S$ is empty,
then $f$ is the all $1$ function. Otherwise let $x$ and $y$ be
any two elements in $S$ (not necessarily distinct). Then if $f$
is $(C_{3},100)$-free, it must be the case that $x+y$ is also in
$S$. Thus $S$ is a linear subspace of $\{0,1\}^{n}$. Suppose the
dimension of $S$ is $k$ with $k\geq1$. Then there are $k$ linearly
independent vectors $a_{1},\ldots,a_{k}\in\cube^{n}$ such that $z\in S$
iff $(\langle z,a_{1}\rangle=0)\bigwedge\cdots\bigwedge(\langle z,a_{k}\rangle=0)$.
Therefore, by De Morgan's law, $f(z)=1$ iff $z\in\bar{S}$ iff $(\langle z,a_{1}\rangle=1)\bigvee\cdots\bigvee(\langle z,a_{k}\rangle=1)$,
which is equivalent to the claim. \end{proof}

\begin{obs}\label{Obs: Obs containment} If $\Sigma\neq1^{k}$ for
some $k>2$, then $\textrm{$(C_{k+2},\Sigma\circ00)$-free}\subseteq\textrm{$(C_{k},\Sigma)$-free}$.
Similarly, if $\Sigma\neq0^{k}$ for some $k>2$, then $\textrm{$(C_{k+2},\Sigma\circ11)$-free}\subseteq\textrm{$(C_{k},\Sigma)$-free}$.
\end{obs} 

\begin{proof} By symmetry, we only need to prove the first part.
Let $f\in\textrm{$(C_{k+2},\Sigma\circ00)$-free}$. Suppose $f\notin\textrm{$(C_{k},\Sigma)$-free}$,
then there exists a violating tuple, say, $\langle x_{1},x_{2},\ldots,x_{j}\rangle$
such that $\sum_{i=1}^{j}x_{i}=0$ and \[
\langle f(x_{1}),f(x_{2}),\ldots,f(x_{j})\rangle=\Sigma.\]
 Since $\Sigma$ is not an all $1$ vector, there exists some $k$
such that $f(x_{k})=0$. But then $\langle x_{1},x_{2},\ldots,x_{j},x_{k},x_{k}\rangle$
would be a violation tuple of pattern $(C_{k+2},\Sigma\circ00)$,
contradicting our assumption that the function $f$ is in $\textrm{$(C_{k+2},\Sigma\circ00)$-free}$.
\end{proof}

\begin{obs}\label{Obs: Obs 0011 constant} $(C_{4},0011)$-free equals
the set of constant functions. \end{obs}

\begin{proof} Clearly a constant function has no $0011$ pattern.
For the reverse inclusion, suppose $f$ is $(C_{4},0011)$-free but
not a constant function. Then there exist $x$ and $y$ such that
$f(x)=0$ $f(y)=1$. Then \[
\langle f(x),f(x),f(y),f(y)\rangle=0011.\]
 \end{proof}

\begin{proof}[Proof of Theorem \ref{thm:characterization}]$\mbox{}$
\begin{enumerate}
\item Follows from Observation~\ref{Obs: Obs containment} and Observation~\ref{Obs: Obs 0011 constant}. 
\item We only need to prove the first half of the claim, the second half
will then follow by symmetry. It is easy to check that, if $O(\Sigma)>1$
is odd and $Z(\Sigma)$ is even, then $\mlin\subseteq\calf_{C_{k},\Sigma}$.
To prove the other containment, note that by Observation~\ref{Obs:Obs 1 char}
and Observation~\ref{Obs: Obs containment}, $\calf_{C_{k},\Sigma}$
is contained in the disjunction of linear functions and in particular,
we may assume $f$ is $(C_{5},00111)$-free. So $f(x)=(\langle a_{1},x\rangle=1)\bigvee(\langle a_{2},x\rangle=1)\bigvee\cdots$,
where $a_{1},a_{2},\ldots,$ are non-zero, distinct and linearly independent
vectors. Since $a_{1}$ and $a_{2}$ are linearly independent, there
exist $x_{1},x_{2}$ such that $\langle a_{1},x_{1}\rangle=\langle a_{2},x_{2}\rangle=1$
while $\langle a_{1},x_{2}\rangle=\langle a_{2},x_{1}\rangle=0$.
Then $\langle f(0),f(0),f(x_{1}),f(x_{2}),f(x_{1}+x_{2})\rangle=00111$.
Therefore $f$ cannot be the disjunction of more than one linear function,
making it linear. Finally note that $\mlin\subseteq\textrm{$(C_{|\Sigma|},\Sigma)$-free}\subseteq\textrm{$(C_{5},00111)$-free}\subseteq\mlin$. 
\item This follows from Observation~\ref{Obs:Obs 1 char} and Observation~\ref{Obs: Obs containment}. 
\item Let $i$ and $j$ be odd integers. If $f(x)$ is linear, then it is
$\textrm{$(C_{i+j},0^{i}1^{j})$-free}$ since $j$ is odd. If $f(x)$
is the complement of linear, then it is $\textrm{$(C_{i+j},0^{i}1^{j})$-free}$
since $i$ is odd. So if $f$ is an affine function, then it is $\textrm{$(C_{i+j},0^{i}1^{j})$-free}$.
Now consider $\textrm{$(C_{6},000111)$-free}$. If $f$ is $\textrm{$(C_{6},000111)$-free}$
then the set $f^{-1}(1)$ forms an affine subspace (since $f$ is
also $\textrm{$(C_{4},0111)$-free}$.). Similarly the set $f^{-1}(0)$
forms an affine subspace (since $f$ is also $\textrm{$(C_{4},0001)$-free}$)
and so $f$ is an affine function. 
\item This follows from Observation~\ref{Obs:Obs 0 char} and Observation~\ref{Obs: Obs containment}. 
\end{enumerate}
\end{proof} 
\end{document}